%% file: main.tex
\documentclass[11pt]{amsart}
\usepackage[T1]{fontenc}
\usepackage[utf8]{inputenc}
\usepackage{lmodern}
\usepackage{amsmath,amssymb,amsthm}
\usepackage[a4paper,margin=30mm]{geometry}
\usepackage[hidelinks]{hyperref}
\usepackage{url}
\input{glyphtounicode}
\newtheorem{theorem}{Theorem}[section]
\newtheorem{lemma}[theorem]{Lemma}
\newtheorem{proposition}[theorem]{Proposition}
\newtheorem{corollary}[theorem]{Corollary}
\theoremstyle{definition}

\theoremstyle{remark}
\newtheorem{remark}[theorem]{Remark}
\DeclareMathOperator{\stack}{stack}
\DeclareMathOperator{\estim}{estim}
\DeclareMathOperator{\dist}{d}
\DeclareMathOperator{\supp}{supp}
\newcommand{\mass}[1]{\lvert #1\rvert}
\newcommand{\NN}{\mathbb N}
\newcommand{\ZZ}{\mathbb Z}
\newcommand{\EMPTY}{\mathsf{EMPTY}}
\newcommand{\StackableAt}{\operatorname{StackableAt}}
\title{The stacking number of a tree}
\author{John Fairfax-Ball}
\thanks{ORCID iD: \url{https://orcid.org/0009-0007-9017-4571}}
\date{}
\hypersetup{pdftitle={The stacking number of a tree},pdfauthor={John Fairfax-Ball}}
\begin{document}
\begin{abstract}
Csern\'ak and Soukup introduced the stacking number of a graph: the least
integer $t\ge2$ such that every configuration of $t$ pebbles can be
transformed by pebbling moves into a configuration supported on one vertex.
They conjectured an explicit formula for the stacking number of a tree in
terms of the degrees of its vertices and their distances from a chosen root.
We prove this conjecture for every finite tree with at least two vertices.
We also show that whether a configuration on a tree can be stacked at a given
vertex is decided exactly by the sign of an integer score built from
recursive branch messages; consequently, stackability can be tested with a
linear number of arithmetic operations. The proof combines this criterion
with an explicit zero-score obstruction and a weighted cancellation argument
for arbitrary nonstackable configurations. The main theorem and the
stackability criterion are formalized in Lean 4, and the main theorem has
passed Palomar mechanical verification; it is publicly registered as
\texttt{PALOMAR-2026-09-25-000010}.
\end{abstract}
\maketitle

\section{Introduction}
A pebbling move removes two pebbles from a vertex and puts one pebble on an
adjacent vertex. Csern\'ak and Soukup~\cite{cs2026} introduced the stacking
number: how many pebbles guarantee that all remaining pebbles can be brought
to a single vertex, whose location may depend on the configuration? They
conjectured an explicit formula for trees and verified it computationally
for every tree of order at most seven.

For a finite tree $T$ and $r\in V(T)$, put
\begin{align}
 \ell_T(r)&=\bigl|\{v\ne r:\deg_T(v)=1\}\bigr|,\label{eq:leaves}\\
 \sigma_T(r)&=1+\sum_{\substack{v\in V(T)\\v=r\text{ or }\deg_T(v)>1}}
                 \deg_T(v)2^{\dist_T(r,v)},\label{eq:sigma}\\
 E_T(r)&=\sigma_T(r)+\ell_T(r),\qquad
 \estim(T)=\max_{r\in V(T)}E_T(r).\label{eq:estimator}
\end{align}
Here $\dist_T$ denotes graph distance. We retain the source notation
$\sigma_T$ and $\estim$; our $\ell_T(r)$ is its $\operatorname{leaf}(r)$,
and $E_T$ abbreviates its rooted estimator.
Conjecture 10.3 of~\cite{cs2026} asserts $\stack(T)=\estim(T)$.

The definition of the stacking number minimizes over $t\ge2$. Thus the
one-vertex tree has stacking number $2$, whereas the displayed estimator is
$1$. We state the conjecture with the explicit hypothesis $|V(T)|\ge2$;
this isolates a boundary convention in the definitions.

\begin{theorem}\label{thm:main}
For every finite tree $T$ with at least two vertices,
\[
 \stack(T)=\estim(T).
\]
\end{theorem}

In a subsequent preprint~\cite{cs2026ash}, Csern\'ak and Soukup prove the
lower bound $\stack(T)\ge\estim(T)$ and derive the matching upper bound from
their Almost Stacked Hypothesis. Theorem~\ref{thm:main} shows that the
equality holds unconditionally. That preprint also normalizes the estimator
as $\max\{2,\estim(T)\}$; with this convention the equality extends to the
one-vertex tree.

The proof yields an exact criterion of independent interest. For a
configuration $C$ and a vertex $r$, Section~\ref{sec:messages} defines an
integer score $S_r(C)$ by a recursion over the branches at $r$, and
Theorem~\ref{thm:root-score} shows that $C$ can be stacked at $r$ if and
only if $S_r(C)>0$. All the scores can be computed together with
$O(|V(T)|)$ arithmetic operations (Corollary~\ref{cor:linear-time}), so the
stackability of a given configuration on a tree can be decided without any
search over sequences of moves.

We prove the tree-stacking estimator conjecture of Csern\'ak and Soukup for
nontrivial trees. The proof has three structural ingredients.
A branch sends an integer message across its boundary edge, recording the
optimal gain or cost of clearing it. An empty branch is a separate state.
These messages give an exact criterion $S_r(C)>0$ for stacking a
configuration $C$ at $r$, including necessity for arbitrary move order.
An explicit configuration of mass $\estim(T)-1$ has every score zero.
For the reverse bound, the nonnegative defects $-S_v(C)$ satisfy local
two-sided edge equations. We assign the excess contributions to distinct
endpoints and cancel them against weighted defects. A consolidation lemma
then bounds the auxiliary height weights by distances from one root.

Classical graph pebbling prescribes a target vertex; see
Chung~\cite{chung1989} and Hurlbert~\cite{hurlbert1999}.
Cover pebbling prescribes positive demands at all vertices. Sj\"ostrand's
theorem~\cite{sjostrand2005} reduces its universal threshold to distance
potentials of configurations concentrated at one vertex.
In stacking, every vertex outside the eventual target must be empty.
The messages below account for this exact clearing condition.

\section{Configurations and the estimator}
All graphs in this paper are finite, simple and undirected. A tree is a
nonempty connected acyclic graph. Unless expressly stated otherwise,
$T$ has at least two vertices. We write $u\sim v$ for adjacency, and omit
the subscript $T$ from degrees and distances when no confusion is possible.
The natural numbers include zero.

A \emph{configuration} is a function $C:V(T)\to\NN$; its mass is
$\mass{C}=\sum_vC(v)$ and its support is $\supp(C)=\{v:C(v)>0\}$.
The source writes $\|c\|$ for mass; we use $\mass{C}$.
The move $u\to v$ is legal when $u\sim v$ and $C(u)\ge2$.
It replaces $C(u)$ by $C(u)-2$ and $C(v)$ by $C(v)+1$.
Reachability permits a finite sequence of legal moves, including the empty
sequence. Each move decreases mass by one.

A configuration is \emph{stacked at $r$} when
$\supp(C)=\{r\}$. Write $\StackableAt(T,C,r)$ if a configuration stacked
at $r$ is reachable from $C$. The configuration is \emph{stackable} if this
holds for some $r$.
In particular, the zero configuration is not stacked or stackable.
The stacking number is the least integer $t\ge2$ for which every
configuration of mass \emph{exactly} $t$ is stackable. Our proof will also
establish the existence of such an integer for trees.

Let $L$ be the number of degree-one vertices of $T$. The estimator has the
following useful form, uniform in the root:
\begin{equation}\label{eq:compact-estimator}
 E_T(r)-1=L+\sum_{\deg(v)>1}\deg(v)2^{\dist(r,v)}.
\end{equation}
If $r$ is internal, this follows directly from the definition. If $r$ is
a leaf, its degree-one term in $\sigma_T(r)$ replaces the leaf omitted from
$\ell_T(r)$. Also
\begin{equation}\label{eq:estim-large}
 E_T(r)\ge |V(T)|+1.
\end{equation}
Indeed, in the expansion
$1+\deg(r)+\ell_T(r)+\sum_{v\ne r,\,\deg(v)>1}\deg(v)2^{\dist(r,v)}$,
the root degree is at least one and each of the other $|V(T)|-1$ vertices
contributes at least one.

\input{sections/messages}
\input{sections/obstruction}
\input{sections/upper}

\section{The stacking formula}\label{sec:formula}
One final point is required because the stacking number uses exact sizes.
\begin{lemma}\label{lem:upward}
Let $G$ be a connected graph on $n\ge2$ vertices. If $t\ge2$ and every
configuration of mass $t$ is stackable, then $t>n$ and every configuration
of mass $t+1$ is stackable.
\end{lemma}
\begin{proof}
If $2\le t\le n$, put one pebble on each of $t$ distinct vertices.
This configuration admits no legal move and is not stacked, a contradiction.
Now let $\mass{C}=t+1$. Since $t+1>n$, some vertex has at least two
pebbles. It has a neighbour by connectedness and $n\ge2$. One legal move
produces a configuration of mass $t$, which is stackable by hypothesis.
Prepend that move to a stacking sequence.
\end{proof}

\begin{proof}[Proof of Theorem~\ref{thm:main}]
Theorem~\ref{thm:mass-bound} implies that every configuration of mass
$\estim(T)$ is stackable. By~\eqref{eq:estim-large}, this is an admissible
threshold at least two, so $\stack(T)\le\estim(T)$.
On the other hand, Corollary~\ref{cor:extremal-obstruction} provides a
nonstackable configuration of mass $\estim(T)-1$. If a universal exact
size $t\ge2$ were smaller than $\estim(T)$, repeated application of
Lemma~\ref{lem:upward} would make every configuration of mass
$\estim(T)-1$ stackable. This contradiction proves the reverse inequality.
\end{proof}

The proof also gives an exact extremal statement:
\begin{corollary}\label{cor:zero-extremal}
For every nontrivial finite tree,
\[
 \max\{\mass{C}:C\text{ is not stackable}\}
 =\max\{\mass{C}:S_v(C)=0\text{ for every }v\}
 =\estim(T)-1.
\]
\end{corollary}
\begin{proof}
Every zero-score configuration is nonstackable by
Theorem~\ref{thm:root-score}. The upper bound is Theorem~\ref{thm:mass-bound},
and the zero-score obstruction attains it.
\end{proof}

For the singleton tree, every positive configuration is already stacked.
Consequently $\stack(K_1)=2$ under the source threshold convention, while
\eqref{eq:leaves}--\eqref{eq:estimator} give $\estim(K_1)=1$.

\section{Formal verification and reproducibility}
The complete theorem, including the message characterization, obstruction,
weighted cancellation and exact-size threshold argument, has been formalized
in Lean. The headline declaration is
\begin{center}
\texttt{TreeStack.stack\_eq\_estim\_of\_two\_le\_card}.
\end{center}
Its hypothesis and conclusion are, respectively,
\[
 2\le\operatorname{Fintype.card}V,
 \qquad
 \texttt{TreeStack.stack}\ T=\texttt{TreeStack.estim}\ T,
\]
for a finite tree $T$ on $V$. The formalization proves the result for
arbitrary finite trees, without an order bound.
Theorem~\ref{thm:root-score} is the declaration
\begin{center}
\texttt{TreeStack.stackableAt\_iff\_score\_pos}.
\end{center}
Corollary~\ref{cor:linear-time} is an algorithmic consequence of that theorem
and is not part of the formal development.

The source and reproduction instructions are available in the public
repository~\cite{treestack}. The immutable formal-verification revision is
\begin{center}
\texttt{4d4969703a9f0ca7a51cbe7edf0f0338cc95cafb}.
\end{center}
It pins Lean 4.35.0-rc2 and a specific Mathlib revision. The theorem-level
audit records only the standard axioms \texttt{propext},
\texttt{Quot.sound} and \texttt{Classical.choice}.
The proved declaration is separate from Palomar's protected Challenge
statement. Palomar mechanical verification succeeded, including comparison
with that statement and independent NanoDa kernel replay. Its automated
review identified no problems. The verified result is publicly registered in
Palomar, version 1, under identifier PALOMAR-2026-09-25-000010. The public
verification record is:
\url{https://palomar-registry.org/entry.html?id=PALOMAR-2026-09-25-000010&version=1}.
This is a record of mechanical verification, not human peer review.

The repository includes a theorem-to-source correspondence for this paper,
the pinned dependency files, and commands for the Lean build and axiom audit.
It also provides independent Python implementations of legal reachability
and branch messages, and bounded regression checks. Those computations
support reproduction and testing; the proofs above do not depend on finite
enumeration. Later publication commits preserve the immutable verification
revision.

\subsection*{Use of AI tools}
This work was developed with extensive assistance from AI tools (\mbox{OpenAI's}
ChatGPT and Codex), under the direction of the author, including
mathematical exploration, Lean development, proof auditing and preparation
of the exposition. The author
selected the problem and scope and directed validation. The proofs in this
paper are written to be checked without the formal development; those of the
main theorem were also cross-checked against the Lean declarations, and the
main theorem is verified in Lean as described above. The author takes full
responsibility for the content of the paper. This disclosure follows the
development provenance recorded in the repository.

\bibliographystyle{amsplain}
\bibliography{references}
\end{document}

%% file: sections/messages.tex
\section{Branch messages and rooted stackability}\label{sec:messages}
For an oriented edge $v\to p$, let $B_{v\mid p}$ be the component
containing $v$ after deleting $vp$. Its root is $v$ and its external boundary
vertex is $p$. The children are the branches $B_{u\mid v}$ with
$u\sim v$, $u\ne p$. Call a branch \emph{occupied} when its restriction of
$C$ is nonzero.

Define the integer transfer function by
\begin{equation}\label{eq:transfer}
 F(x)=\begin{cases}
 2x-3,&x\le1,\\
 x/2,&x\ge2\text{ even},\\
 (x-3)/2,&x\ge3\text{ odd}.
 \end{cases}
\end{equation}
In particular, $F(2)=1$ and $F(3)=0$.
Messages take values in the disjoint union $\{\EMPTY\}\sqcup\ZZ$.
For an empty branch put $m_{v\to p}=\EMPTY$. For an occupied branch put
\begin{equation}\label{eq:message}
 m_{v\to p}=d_{v\to p}:=
 F\left(C(v)+\sum_{\substack{u\sim v,\ u\ne p\\B_{u\mid v}\text{ occupied}}}
                          d_{u\to v}\right).
\end{equation}
The recursion terminates because each child branch has fewer vertices.
Only integer messages are summed. When writing a sum over all neighbours
we use the numerical contribution $\bar d_{u\to v}$, equal to
$d_{u\to v}$ on occupied branches and $0$ on empty branches. This
convention for sums does not identify the states $\EMPTY$ and $0$.
The rooted score is
\begin{equation}\label{eq:score}
 S_r(C)=C(r)+\sum_{u\sim r}\bar d_{u\to r}.
\end{equation}

\begin{theorem}[Exact boundary invariant]\label{thm:boundary}
Let $B=B_{v\mid p}$ be occupied with message $d=d_{v\to p}$.
Work on the subtree induced by $B\cup\{p\}$, with initial configuration
$C\restriction B$ on $B$ and $q\ge0$ pebbles at $p$.
A positive stack at $p$ with $B$ empty is reachable if and only if
$q+d>0$. In that case the greatest possible final pile at $p$ is $q+d$.
\end{theorem}

We prove necessity in a form that also applies when other branches
interleave their moves at $p$.
\begin{lemma}[Boundary flux]\label{lem:flux}
Suppose a legal sequence clears $B_{v\mid p}$. Let $m_{a b}$ count its
moves from $a$ to $b$, and put
$\phi_{v\mid p}=m_{v p}-2m_{p v}$.
If the branch is occupied initially, then
\[
 \phi_{v\mid p}\le d_{v\to p},\qquad
 \phi_{v\mid p}\equiv d_{v\to p}\pmod3.
\]
If it is empty initially, then $\phi_{v\mid p}=-3k$ for some $k\ge0$.
Moves elsewhere in the tree are permitted.
\end{lemma}
\begin{proof}
The move counts satisfy, at every cleared vertex $w$,
\begin{equation}\label{eq:balance}
 C(w)+\sum_{z\sim w}m_{z w}-2\sum_{z\sim w}m_{w z}=0.
\end{equation}
For each initially occupied branch which is cleared, at least one move
crosses its boundary outwards. Indeed, consider the move after which that
branch becomes empty for the last time. An internal move leaves a pebble
inside, and an inward move also does so. Thus this move is an outward
boundary move. This observation applies to every occupied child branch,
independently of the moves outside it.

First consider an initially empty branch. Sum~\eqref{eq:balance} over
its vertices. If $I$ is the number of internal moves and
$a=m_{v p}$, $b=m_{p v}$, the result is $b-2a-I=0$.
Hence its boundary flux is
$a-2b=-3a-2I\le0$. It is divisible by three: give the boundary vertex
sign $+1$ and alternate signs along tree edges. The signed sum of the
balance equations over the branch is congruent to minus its boundary
flux modulo three, because each internal move contributes a multiple of
three. The initial and final signed sums on the empty branch vanish.
This proves the assertion for empty branches, even when they are used
temporarily.

For an occupied branch, induct on its size. By induction, each occupied
child has flux $d_i-3k_i$ with $k_i\ge0$, and each empty child has flux
$-3k_i$. Put $x=C(v)+\sum d_i$ and $K=\sum k_i$, where the first sum
runs over occupied children. At $v$, equation~\eqref{eq:balance} reads
\[
 y+b-2a=0,\qquad y=x-3K,\quad a=m_{v p}\ge1,\quad b=m_{p v}\ge0.
\]
For any integer $y$, these conditions imply
\begin{equation}\label{eq:one-vertex-flux}
 a-2b=2y-3a\le F(y),\qquad a-2b\equiv F(y)\pmod3.
\end{equation}
To see this, the least possible $a$ is
$\max(1,\lceil y/2\rceil)$; substituting this value gives precisely
the three cases of~\eqref{eq:transfer}. Any increase in $a$ lowers
the flux by three.

The transfer function satisfies
\begin{equation}\label{eq:three-loss}
 F(z-3)\le F(z),\qquad F(z)\equiv2z\pmod3.
\end{equation}
For the inequality, the difference $F(z)-F(z-3)$ is $6$ when $z\le1$,
$6,3,3$ at $z=2,3,4$, respectively, and is $0$ for odd $z\ge5$ and
$3$ for even $z\ge6$. The congruence follows from each case in
\eqref{eq:transfer}. Thus $F(x-3K)\le F(x)$ and
$F(x-3K)\equiv F(x)\pmod3$. Apply~\eqref{eq:one-vertex-flux}
to finish the induction.
\end{proof}

\begin{proof}[Proof of Theorem~\ref{thm:boundary}]
The final boundary pile in any allowed clearing sequence is $q+\phi$,
so Lemma~\ref{lem:flux} proves necessity and the upper bound.
For attainment we induct on the branch size.

We use the following scheduling observation. Suppose disjoint occupied
child branches have gains $d_i$, and a vertex initially has $a\ge0$
pebbles. If $a+\sum_i d_i>0$, process branches with positive gain first,
then those with zero gain, then those with negative gain. A positive-gain
task has positive final pile even when begun at zero. Before a zero-gain
task the pile is positive, since the total remaining gains are nonpositive
and the final total is positive. During the negative-gain tasks the pile
after each task is at least the final positive total. The inductive
boundary invariant therefore realizes every task and leaves exactly
$a+\sum_i d_i$. Tasks touch only their branch and the common boundary;
no other task's initial interior configuration is changed.

Let $x=C(v)+\sum_i d_i$ for the occupied children of $B$.
If $x\ge2$, the observation clears the children and leaves $x$ at $v$.
For $x=2k$, make $k$ moves from $v$ to $p$; the final boundary pile is
$q+k=q+F(x)$.
For $x=2k+1\ge3$, make $k$ moves from $v$ to $p$, leaving one pebble
at $v$. The assumption $q+F(x)=q+k-1>0$ implies $q+k\ge2$.
One move $p\to v$ followed by one move $v\to p$ is therefore legal;
it clears $v$ and leaves $q+k-1$ at $p$.

If $x\le1$, put $j=2-x\ge1$. The assumption
$q+F(x)=q-2j+1>0$ implies $q\ge2j$. Make $j$ moves $p\to v$.
The children's effective total at $v$ is now $x+j=2$.
The observation clears them leaving two pebbles at $v$, and one final
move $v\to p$ gives pile $q-2j+1=q+F(x)$.
These constructions also cover the case with no children and complete
the induction.
\end{proof}

\begin{theorem}[Exact rooted stackability]\label{thm:root-score}
For every configuration $C$ on a finite tree and every vertex $r$,
\[
 \StackableAt(T,C,r)\quad\Longleftrightarrow\quad S_r(C)>0.
\]
Consequently $C$ is nonstackable if and only if all its scores are
nonpositive.
\end{theorem}
\begin{proof}
If the score is positive, apply the scheduling observation to the occupied
branches incident with $r$, starting with $C(r)$ pebbles. It gives a stack
of size exactly $S_r(C)$ at $r$.
Conversely, in any sequence ending in a stack at $r$, every incident
branch is cleared. Lemma~\ref{lem:flux} bounds the flux from each occupied
branch by its message and from each empty branch by zero. The final pile
is therefore at most $S_r(C)$, and it is positive.
The final equivalence follows by quantifying over all targets.
\end{proof}

\begin{remark}
An occupied one-vertex branch containing three pebbles has message
$F(3)=0$. It can be cleared to its boundary only when the boundary has a
positive initial pile. An empty branch may be left alone even when that
pile is zero. This is why the recursive state must retain $\EMPTY$
separately from the integer zero.
\end{remark}

\begin{corollary}[Computing stackability]\label{cor:linear-time}
Let $T$ be a tree with $n$ vertices and let $C$ be a configuration on $T$.
All messages $m_{v\to p}$ and all scores $S_r(C)$, $r\in V(T)$, can be
computed with $O(n)$ arithmetic operations. Consequently, $O(n)$ arithmetic
operations suffice to decide whether $C$ is stackable and to list every
vertex at which it is stackable.
\end{corollary}
\begin{proof}
Fix a vertex $\rho$ and list the vertices of $T$ in breadth-first order from
$\rho$. Scanning this list backwards, determine for each $v\ne\rho$, with
parent $p$, whether $B_{v\mid p}$ is occupied and, if it is, compute
$d_{v\to p}$ by~\eqref{eq:message}. The inputs are $C(v)$ and the states
of the branches $B_{u\mid v}$ at the children $u$ of $v$, all found earlier
in the scan.

Next scan the list forwards. When a vertex $p$ is reached, the states of all
branches $B_{u\mid p}$ with $u\sim p$ are known: those at the children of $p$
from the first scan, and the one at its parent earlier in this scan. Put
\[
 \Sigma(p)=C(p)+\sum_{u\sim p}\bar d_{u\to p},
 \qquad
 \omega(p)=\bigl|\{u\sim p:B_{u\mid p}\text{ occupied}\}\bigr|.
\]
Let $v$ be a child of $p$. The branch $B_{p\mid v}$ is occupied exactly when
$C(p)>0$ or $B_{u\mid p}$ is occupied for some neighbour $u\ne v$ of $p$;
the number of such neighbours is $\omega(p)-1$ if $B_{v\mid p}$ is occupied
and $\omega(p)$ otherwise. If $B_{p\mid v}$ is occupied,
then~\eqref{eq:message} gives
\[
 d_{p\to v}=F\bigl(\Sigma(p)-\bar d_{v\to p}\bigr).
\]
At each vertex $v$ each scan uses $O(\deg(v)+1)$ arithmetic operations, so
both scans together use $O(n)$. Finally $S_r(C)=\Sigma(r)$
by~\eqref{eq:score}, and Theorem~\ref{thm:root-score} converts the signs of
the scores into the stated decisions.
\end{proof}

%% file: sections/obstruction.tex
\section{An extremal zero-score configuration}\label{sec:obstruction}
Fix $r\in V(T)$ and root the tree at $r$. Define
\begin{equation}\label{eq:obstruction}
 C_r(v)=\begin{cases}
 \sigma_T(r)-1,&v=r,\\
 1,&v\ne r\text{ and }\deg(v)=1,\\
 0,&\text{otherwise}.
 \end{cases}
\end{equation}
For a nonroot vertex $v$, write $T_v$ for its descendant subtree and put
\begin{equation}\label{eq:obstruction-height}
 H_v=\begin{cases}
 1,&v\text{ has no children},\\
 3+2\sum_{u\text{ child of }v}H_u,&\text{otherwise}.
 \end{cases}
\end{equation}
These are positive integers, with closed form
\begin{equation}\label{eq:height-closed}
 H_v=1+2\sum_{\substack{w\in V(T_v)\\\deg_T(w)>1}}
                      \deg_T(w)2^{\dist(v,w)}.
\end{equation}
For a leaf this is immediate. For a nonleaf $v\ne r$, the number of
children is $\deg_T(v)-1$. Substitution of the formula at its children
gives
\[
 3+2\sum_uH_u
 =1+2\deg_T(v)+4\sum_u\sum_{\substack{w\in V(T_u)\\\deg_T(w)>1}}
                        \deg_T(w)2^{\dist(u,w)},
\]
which is~\eqref{eq:height-closed} at $v$.
Summing this closed form over the children of $r$ yields
\begin{equation}\label{eq:height-root}
 \sum_{u\text{ child of }r}H_u
 =\deg_T(r)+\sum_{\substack{w\ne r\\\deg_T(w)>1}}
                         \deg_T(w)2^{\dist(r,w)}
 =\sigma_T(r)-1.
\end{equation}

\begin{theorem}\label{thm:obstruction}
For the configuration $C_r$,
\[
 S_v(C_r)=0\quad(v\in V(T)),\qquad \mass{C_r}=E_T(r)-1.
\]
In particular, $C_r$ is nonstackable.
\end{theorem}
\begin{proof}
Every descendant branch is occupied, since it contains a leaf other than
$r$. The opposite side is occupied as well, since
$C_r(r)=\sigma_T(r)-1\ge\deg_T(r)\ge1$.
Induction from the leaves gives upward messages $d_{v\to p}=-H_v$:
at a leaf this is $F(1)=-1$; at an internal nonroot vertex the effective
input is $-\sum_uH_u$, so its message is
$-3-2\sum_uH_u=-H_v$. Equation~\eqref{eq:height-root} implies
$S_r(C_r)=0$.

Now suppose $p$ has score zero and $v$ is its child. Removing the
contribution $-H_v$ from its score gives effective input $H_v$ on the
$p$-side of $pv$, hence $d_{p\to v}=F(H_v)$.
If $v$ is a leaf, its score is $1+F(1)=0$. Otherwise let
$K=\sum_{u\text{ child of }v}H_u\ge1$.
Then $H_v=3+2K$, $F(H_v)=K$, and the score at $v$ is $-K+K=0$.
Propagation away from $r$ proves the assertion at every vertex.
Finally, the total mass is
$\sigma_T(r)-1+\ell_T(r)=E_T(r)-1$.
Nonstackability follows from Theorem~\ref{thm:root-score}.
\end{proof}

\begin{corollary}\label{cor:extremal-obstruction}
There is a nonstackable configuration of mass $\estim(T)-1$ whose score
is zero at every vertex.
\end{corollary}
\begin{proof}
Choose a root maximizing $E_T$ and apply Theorem~\ref{thm:obstruction}.
\end{proof}

%% file: sections/upper.tex
\section{The arbitrary-defect upper bound}\label{sec:upper}

We now bound the mass of every configuration whose scores are all
nonpositive.  All degrees, paths and distances in this section belong to the
original tree.  In particular, the argument retains the empty branches:
their contributions will be cancelled separately from the contributions of
edges with occupied support on both sides.

Let $T$ be a finite tree with at least two vertices, and let $C$ satisfy
$S_v(C)\leq 0$ for every $v$.  Set
\[
 \delta_v=-S_v(C)\geq 0.
\]
If $C$ is identically zero the desired bound is immediate.  Henceforth choose
an occupied vertex $o$, so that $C(o)>0$.  This vertex is used to assign edge
owners; the estimator root obtained at the end may be different.

\subsection{Retained edges and the local defect equations}

For adjacent vertices $u,v$, write $T_{u\mid v}$ for the component containing
$u$ after deleting $uv$.  Call $uv$ \emph{retained} if both
$T_{u\mid v}$ and $T_{v\mid u}$ contain an occupied vertex.  Its two messages
are then integers.  For any edge, let
\[
 \overline d_{u\to v}=
 \begin{cases}
 d_{u\to v},&T_{u\mid v}\text{ is occupied},\\
 0,&T_{u\mid v}\text{ is empty}.
 \end{cases}
\]
This notation records the numerical contribution of a message in a score;
it does not identify the message $\EMPTY$ with the integer message
zero.  Thus
\begin{equation}\label{eq:upper-score}
 C(v)=-\delta_v-\sum_{u\sim v}\overline d_{u\to v}.
\end{equation}

\begin{lemma}[Support separation]\label{lem:upper-support}
The two sides of an edge partition $V(T)$.  Every nonretained edge has
exactly one empty side.  A retained edge lies on a path between two occupied
vertices, and every edge of a path between two occupied vertices is retained.
More specifically, if $uv$ is retained and $x$ is an occupied vertex in
$T_{u\mid v}$, every edge on the path from $u$ to $x$ is retained.  If
$uv$ is retained and no other retained edge is incident with $u$, then
$C(u)>0$.
Finally, an endpoint on the empty side of an edge is incident with no
retained edge.
\end{lemma}

\begin{proof}
Deleting an edge of a tree produces exactly two components, giving the
partition.  Since $C(o)>0$, at least one component is occupied.  If both
are occupied, choose one occupied vertex in each; their unique path crosses
the deleted edge.  Conversely, deleting any edge on the path between two
occupied vertices separates those vertices, so that edge is retained.

For the more specific assertion, choose an occupied vertex $y$ in
$T_{v\mid u}$.  The path from $x$ to $y$ consists of the path from $x$ to
$u$, the edge $uv$, and the path from $v$ to $y$.  Every edge on its
$x$--$u$ segment is therefore retained.  If $u$ is unoccupied, an occupied
vertex $x$ on its side is distinct from $u$, and the edge of that
segment incident with $u$ is a retained edge other than $uv$.  This proves
the next assertion by contraposition.

For the last assertion, let $T_{v\mid u}$ be empty.  The edge $uv$ itself
is nonretained.  For every other neighbour $w$ of $v$, the component
$T_{w\mid v}$ lies inside $T_{v\mid u}$ and is empty.  Hence $wv$ is
also nonretained.
\end{proof}

For a retained edge $uv$, put $a=d_{u\to v}$ and $b=d_{v\to u}$.
Removing the opposite message from each endpoint score gives the two
effective branch inputs.  The message recursion therefore yields
\begin{equation}\label{eq:upper-edge-equations}
 a=F(-\delta_u-b),\qquad b=F(-\delta_v-a).
\end{equation}
These equations are only asserted for retained edges.

\begin{lemma}[Classification of retained edge states]\label{lem:upper-classification}
The solutions of \eqref{eq:upper-edge-equations} with
$\delta_u,\delta_v\geq0$ have exactly the following forms, up to exchanging
$u$ and $v$:
\begin{enumerate}
 \item An \emph{oriented state}, with
 \begin{equation}\label{eq:upper-oriented}
 a=k\geq0,\qquad b=-2(k+\delta_v)-3,
 \qquad
 \delta_u=2\delta_v\ \text{or}\ \delta_u=2\delta_v+3.
 \end{equation}
 The latter alternative requires $k>0$.  We direct this edge $u\to v$.
 \item A \emph{two-negative state}, with unique $r,q\in\NN$ such that
 \begin{equation}\label{eq:upper-negative}
 a=-(2r+1),\quad b=-(2q+1),\quad
 \delta_u=r+2q,\quad\delta_v=q+2r.
 \end{equation}
\end{enumerate}
Both messages cannot be nonnegative.
\end{lemma}

\begin{proof}
The definition of $F$ gives
\begin{align*}
 F(z)<0&\iff z\leq1,\\
 F(z)=-(2r+1)&\iff z=1-r\qquad(r\geq0),\\
 F(z)=k&\iff
   \bigl(k>0\text{ and }z=2k\bigr)\text{ or }z=2k+3
   \qquad(k\geq0).
\end{align*}
Indeed, the first two statements use $F(z)=2z-3$ for $z\leq1$;
the even positive and odd positive cases of $F$ give the third, including
the unique preimage $3$ of zero.

If $a=k\geq0$, then $-\delta_v-k\leq0$, so the second edge equation
gives $b=-2(k+\delta_v)-3<0$.  Substituting this into the first equation
and using the nonnegative preimages gives either
$-\delta_u+2k+2\delta_v+3=2k+3$, or the same left side equals $2k$
with $k>0$.  These are exactly the two defect alternatives in
\eqref{eq:upper-oriented}.  The case $b\geq0$ is symmetric.

If both messages are negative, the first two identities express them
uniquely as $-(2r+1)$ and $-(2q+1)$, with $r,q\geq0$.
Their unique preimages give
$-\delta_u+2q+1=1-r$ and
$-\delta_v+2r+1=1-q$, proving \eqref{eq:upper-negative}.
Conversely, substitution verifies every displayed state, subject to the
stated restriction when $k=0$.
\end{proof}

Call an oriented state \emph{defective} if its head defect $\delta_v$ is
positive.  Call a two-negative state defective if $r+q>0$.  The remaining
two-negative state is the neutral state $(-1,-1)$, with both endpoint
defects zero.  The defective retained edges form a forest, since they are
a subset of the edges of $T$.

\subsection{Owners, orientations and auxiliary weights}

\begin{lemma}[Injective incident owners]\label{lem:upper-owners}
For each edge $e=uv$, let $\omega(e)$ be the endpoint farther from $o$.
Then $\omega$ is well defined and injective on $E(T)$.  If
$T_{v\mid u}$ is empty, then $\omega(uv)=v$.
\end{lemma}

\begin{proof}
The unique path from $o$ crosses $uv$ precisely when reaching the endpoint
on the side not containing $o$.  Thus the endpoint distances differ by
one.  The owner is the more distant endpoint.  Each vertex other than $o$
has exactly one neighbour preceding it on its unique path from $o$;
consequently it owns exactly that predecessor edge.  In particular two
edges cannot share an owner.  If $T_{v\mid u}$ is empty, the occupied
vertex $o$ lies on the $u$ side, so $v$ is the farther endpoint.
\end{proof}

Define an auxiliary partial orientation of $T$ as follows.  Keep the direction
of every oriented state, including those with zero head defect.  Direct each
defective two-negative edge toward its owner.  Leave neutral edges and
nonretained edges unoriented.  The classification ensures that no edge has
both directions.  Since $T$ has no undirected cycle, this directed graph has
no directed cycle.

Let $h(v)$ be the maximum length of a directed path ending at $v$, allowing
the path of length zero, and set
\begin{equation}\label{eq:upper-weights}
 A_v=2^{h(v)}.
\end{equation}
The maximum exists because a directed path has no repeated vertex.
For an auxiliary arrow $u\to v$, a longest path ending at $u$ can be
extended by this arrow: if it already contained $v$, it would give a
directed cycle.  Hence
\begin{equation}\label{eq:upper-doubling}
 h(v)\geq h(u)+1,\qquad A_v\geq2A_u.
\end{equation}
A vertex has positive height exactly when it has an incoming auxiliary
arrow.  Since every auxiliary arrow is retained, Lemma~\ref{lem:upper-support}
also shows that the endpoint on the empty side of any edge has height zero
and weight one.

\subsection{Local charges and global defect cancellation}

For an undirected edge $e=uv$, define its weighted message contribution by
\[
 E_e=-\bigl(A_v\overline d_{u\to v}
             +A_u\overline d_{v\to u}\bigr).
\]
Define two nonnegative charges, with disjoint conditions of application:
\begin{align*}
 Q_e^{\mathrm{ret}}
  &=\begin{cases}
     A_{\omega(e)}\delta_{\omega(e)},&e\text{ is retained and defective},\\
     0,&\text{otherwise},
    \end{cases}\\
 Q_e^{\mathrm{emp}}
  &=\begin{cases}
     A_{\omega(e)}\delta_{\omega(e)},&e\text{ is nonretained},\\
     0,&\text{otherwise}.
    \end{cases}
\end{align*}

\begin{lemma}[Retained-edge charge]\label{lem:upper-retained-charge}
For every retained edge $uv$,
\[
 E_{uv}\leq A_u+A_v+Q_{uv}^{\mathrm{ret}}.
\]
\end{lemma}

\begin{proof}
First consider an oriented state $u\to v$ with parameter $k\geq0$.
By \eqref{eq:upper-oriented} and \eqref{eq:upper-doubling},
\begin{align*}
 E_{uv}
 &=A_u\bigl(2(k+\delta_v)+3\bigr)-A_vk\\
 &=3A_u+k(2A_u-A_v)+2A_u\delta_v\\
 &\leq A_u+A_v+2A_u\delta_v.
\end{align*}
The excess $2A_u\delta_v$ is zero when the edge is nondefective.
Otherwise either endpoint can pay it: the local defect relation gives
$A_u\delta_u\geq2A_u\delta_v$, and doubling gives
$A_v\delta_v\geq2A_u\delta_v$.  Thus it is bounded by the charge at
the prescribed owner, regardless of whether that owner is the tail or head.

For a defective two-negative state, name its endpoints so that the auxiliary
direction is $u\to v$, with $v=\omega(uv)$, and use the parameters
in \eqref{eq:upper-negative}.  Then
\begin{align*}
 E_{uv}
 &=A_u+A_v+2A_vr+2A_uq\\
 &\leq A_u+A_v+A_v(2r+q)\\
 &=A_u+A_v+A_v\delta_v.
\end{align*}
Here $q\geq0$ and $2A_u\leq A_v$ justify the inequality.  A neutral
edge has $r=q=0$ and contributes exactly $A_u+A_v$, completing the proof.
\end{proof}

\begin{lemma}[Empty-side cancellation]\label{lem:upper-empty-charge}
If $T_{v\mid u}$ is empty, then
\[
 E_{uv}=A_v\delta_v=Q_{uv}^{\mathrm{emp}}.
\]
\end{lemma}

\begin{proof}
The message $m_{v\to u}$ is literally $\EMPTY$.  At $v$,
the initial pile is zero, and every incident branch other than the $u$ side
is empty.  Consequently
$S_v(C)=d_{u\to v}$, where this last message is an integer because the
$u$ side contains $o$.  Thus
$\overline d_{v\to u}=0$ and $d_{u\to v}=-\delta_v$, giving
$E_{uv}=A_v\delta_v$.  By Lemma~\ref{lem:upper-owners}, $v$ is the
owner.  In addition $A_v=1$, as observed above.  No integer defect-state
equations were applied to the empty message.
\end{proof}

\begin{proposition}[Weighted cancellation]\label{prop:upper-weighted}
The auxiliary weighted mass satisfies
\begin{equation}\label{eq:upper-weighted-bound}
 \sum_v A_vC(v)\leq\sum_v\deg_T(v)A_v.
\end{equation}
\end{proposition}

\begin{proof}
Grouping \eqref{eq:upper-score} by undirected edges gives the exact identity
\begin{equation}\label{eq:upper-weighted-identity}
 \sum_v A_vC(v)=\sum_{e\in E(T)}E_e-\sum_v A_v\delta_v.
\end{equation}
Lemmas~\ref{lem:upper-retained-charge} and \ref{lem:upper-empty-charge}
give, uniformly over all ambient edges,
\[
 E_{uv}\leq A_u+A_v+Q_{uv}^{\mathrm{ret}}+Q_{uv}^{\mathrm{emp}}.
\]
Each sum of the two charges is at most its one owner budget
$A_{\omega(e)}\delta_{\omega(e)}$.  Owners are injective on all edges,
and every vertex budget is nonnegative, so
\[
 \sum_e\bigl(Q_e^{\mathrm{ret}}+Q_e^{\mathrm{emp}}\bigr)
 \leq\sum_e A_{\omega(e)}\delta_{\omega(e)}
 \leq\sum_v A_v\delta_v.
\]
Substitution in \eqref{eq:upper-weighted-identity} cancels the budgets.
Finally, each vertex weight occurs once for each incident edge, so
$\sum_{uv\in E(T)}(A_u+A_v)=\sum_v\deg_T(v)A_v$.
\end{proof}

\subsection{Removing leaf weights}

Let $L$ be the number of ambient vertices of degree one.

\begin{lemma}[Ambient leaf slack]\label{lem:upper-leaf-slack}
Every degree-one vertex $v$ with $A_v>1$ is occupied.  Consequently
\begin{equation}\label{eq:upper-compact}
 \mass{C}\leq L+\sum_{\deg_T(v)>1}\deg_T(v)2^{h(v)}.
\end{equation}
\end{lemma}

\begin{proof}
If $A_v>1$, then $h(v)>0$, so there is an incoming auxiliary arrow at
$v$.  Its underlying edge is retained.  As $v$ has degree one, its side
of that edge is the singleton $\{v\}$; retention implies $C(v)\geq1$.
For any degree-one vertex, therefore,
\[
 A_v-1\leq (A_v-1)C(v):
\]
if $A_v=1$ both sides are zero, and otherwise $C(v)\geq1$ proves it.
Since $A_v\geq1$ at every vertex, all configuration slack terms are
nonnegative.  Using Proposition~\ref{prop:upper-weighted},
\begin{align*}
 \mass{C}
 &=\sum_v A_vC(v)-\sum_v(A_v-1)C(v)\\
 &\leq\sum_v\deg_T(v)A_v
       -\sum_{\deg_T(v)=1}(A_v-1)\\
 &=L+\sum_{\deg_T(v)>1}\deg_T(v)A_v.
\end{align*}
The last equality uses that every vertex of a nontrivial tree has positive
degree.  Replacing $A_v$ by $2^{h(v)}$ proves the claim.  Thus occupied
leaves pay the required slack, while unoccupied ambient leaves have unit
weight and require no payment.
\end{proof}

\subsection{Consolidation to one ambient root}

We prove a pointwise statement about heights before applying the estimator.
A set of vertices is called connected here if the unique tree path between
any two of its vertices stays in the set.  This agrees with connectedness
of the induced subgraph, by uniqueness of paths in a tree.

\begin{lemma}[Source partition]\label{lem:upper-source-partition}
The vertices of $T$ can be partitioned into connected sets $P_i$ with
vertices $s_i\in P_i$ such that
\[
 h(v)=\dist_T(s_i,v)\qquad(v\in P_i).
\]
\end{lemma}

\begin{proof}
For every $v$ of positive height, choose once and for all an incoming
auxiliary arrow $p(v)\to v$ with $h(p(v))=h(v)-1$.  Such an arrow
exists by taking the last edge of a longest directed path ending at $v$:
its prefix has length $h(v)-1$, and \eqref{eq:upper-doubling} at the level
of heights prevents the height of its penultimate vertex from being larger.

Iterating $p$ decreases height by exactly one each time and ends at a
vertex $s(v)$ of height zero after $h(v)$ steps.  No vertex repeats, so
this predecessor chain is the unique tree path from $v$ to $s(v)$ and
$h(v)=\dist_T(s(v),v)$.  Every vertex on the chain has the same terminal
source, since its subsequent predecessor choices are fixed.  Each nonempty
fibre $P_s=\{v:s(v)=s\}$ therefore contains $s$ and the full path from
each of its vertices to $s$.  The path between two vertices of $P_s$ is
contained in the union of their paths to $s$, hence stays in $P_s$.
The nonempty fibres give the required connected partition.
\end{proof}

\begin{lemma}[Merging height-dominated parts]\label{lem:upper-merge}
Let $A,B$ be disjoint nonempty connected vertex sets joined by an edge
$xy$, with $x\in A$ and $y\in B$.  Suppose $a\in A$, $b\in B$ and
$g:V(T)\to\NN$ satisfy
\[
 g(v)\leq\dist_T(a,v)\quad(v\in A),\qquad
 g(v)\leq\dist_T(b,v)\quad(v\in B).
\]
Then $A\cup B$ is connected, and one of $a,b$ satisfies the same domination
on all of $A\cup B$.
\end{lemma}

\begin{proof}
Paths within the two sets, together with $xy$, connect their union.
If $z\in A$ and $w\in B$, concatenate the path from $z$ to $x$ in
$A$, the edge $xy$, and the path from $y$ to $w$ in $B$.  Disjointness
makes this a simple path, so it is the unique tree path and
\begin{equation}\label{eq:upper-boundary-distance}
 \dist_T(z,w)=\dist_T(z,x)+1+\dist_T(y,w).
\end{equation}
Put $\alpha=\dist_T(a,x)$ and $\beta=\dist_T(b,y)$.
If $\beta\leq\alpha+1$, then for every $v\in B$ the triangle
inequality and \eqref{eq:upper-boundary-distance} give
\[
 g(v)\leq\dist_T(b,v)
 \leq\beta+\dist_T(y,v)
 \leq\alpha+1+\dist_T(y,v)=\dist_T(a,v).
\]
The pre-existing domination on $A$ is unchanged, so $a$ works.
Otherwise $\beta>\alpha+1$, which implies $\alpha\leq\beta+1$;
the same argument with the two sets exchanged shows that $b$ works.
\end{proof}

\begin{proposition}[Height-dominating root]\label{prop:upper-consolidation}
There exists a vertex $r\in V(T)$ such that
\begin{equation}\label{eq:upper-height-domination}
 h(v)\leq\dist_T(r,v)\qquad(v\in V(T)).
\end{equation}
\end{proposition}

\begin{proof}
Begin with the connected source partition of
Lemma~\ref{lem:upper-source-partition}; each part is height-dominated by
its own source.  If there is more than one part, connectedness of $T$
provides an edge joining two parts: for example, follow a path from a
vertex in one part to a vertex outside it and take its first exit edge.
Lemma~\ref{lem:upper-merge}, with $g=h$, merges those two parts and
supplies a root that height-dominates their union.  The resulting sets are
still a partition into connected, height-dominated parts.  Each merge
reduces the number of parts by one, so finitely many merges leave the
whole vertex set and one dominating root.
\end{proof}

\begin{theorem}[Arbitrary-defect upper bound]\label{thm:mass-bound}
For every configuration $C$ on a finite tree $T$ with at least two vertices,
\[
 \bigl(S_v(C)\leq0\text{ for all }v\bigr)
 \quad\Longrightarrow\quad
 \mass{C}\leq\estim(T)-1.
\]
\end{theorem}

\begin{proof}
For $C=0$, the bound follows from the positive estimator formula.  Otherwise
the preceding construction applies.  Choose the root $r$ from
Proposition~\ref{prop:upper-consolidation}.  Since $2^x$ increases with
$x$, Lemma~\ref{lem:upper-leaf-slack} gives
\begin{align*}
 \mass{C}
 &\leq L+\sum_{\deg_T(v)>1}\deg_T(v)2^{h(v)}\\
 &\leq L+\sum_{\deg_T(v)>1}\deg_T(v)2^{\dist_T(r,v)}\\
 &\leq L+\max_{s\in V(T)}
              \sum_{\deg_T(v)>1}\deg_T(v)2^{\dist_T(s,v)}
   =\estim(T)-1.
\end{align*}
The final equality is the estimator identity~\eqref{eq:compact-estimator}.
Every quantity in this calculation uses the ambient tree.
\end{proof}

Together with the exact score criterion (Theorem~\ref{thm:root-score}), Theorem~\ref{thm:mass-bound}
implies that every non-stackable configuration has mass at most
$\estim(T)-1$, and hence every configuration of mass $\estim(T)$ is
stackable.